\documentclass{article}
\usepackage{amsthm,amssymb, amsmath, amsfonts}
\usepackage{verbatim}
\usepackage{braket}
\usepackage{cancel}
\usepackage{enumitem}
\usepackage{mathrsfs}
\usepackage[english]{babel}
\usepackage{datetime}

\setenumerate[1]{label={(\arabic*)}}

\newcommand{\ret}{\vartriangleleft}
\newcommand{\emb}{\hookrightarrow}
\newcommand{\bb}[1] {\mathbb{#1}}

\theoremstyle{definition}
\newtheorem{defn}{Definition}[section]

\theoremstyle{plain}
\newtheorem{lemma}{Lemma}
\newtheorem{theorem}{Theorem}

\title{\textbf{Characterizing Bipartite Graphs which Admit a $k$-NU Polymorphism via Absolute Retracts}}
\author{
Adam Jaffe \\
\footnotesize{\texttt{ajaffe@uwaterloo.ca}} \\
\small{\textit{University of Waterloo, 200 University Avenue W., Waterloo, ON, Canada, N2L 3G1}}
}
\date{}

\begin{document}

\maketitle

\textbf{Abstract:} We first introduce the class of bipartite absolute retracts with respect to tree
obstructions with at most $k$ leaves.  Then, using the theory of
homomorphism duality, we show that this class of absolute
retracts coincides exactly with the bipartite graphs
which admit a $(k+1)$-ary near-unanimity polymorphism.  This result
mirrors the case for reflexive graphs and generalizes
a known result for bipartite graphs which admit a $3$-ary near-unanimity
polymorphism.

\vspace{2mm}
\small{\textbf{Keywords:} Absolute retract, Duality, Near-unanimity polymorphism}

\vspace{2mm}

\small{\textbf{AMS subject classifications:} 05C75, 08B05}

\section{Introduction}

Given a digraph $\bb{H}$ with vertex set $H$, a \textit{k-ary near-unanimity polymorphism},
or $k$-NU polymorphism,
is a graph homomorphism $f: \bb{H}^k \rightarrow \bb{H}$ that satisfies the
near-unanimity property
\[
	f(x, \dots, x) = f(y,x,x,\dots,x) = f(x,y,x,\dots,x) = \dots = f(x,x,x,\dots,y) = x,
\]
for all vertices $x,y \in H$.
The problem of determining
which digraphs admit a near-unanimity polymorphism of some arity
has been studied both for reflexive digraphs, where every vertex has a loop, 
and for simple graphs.
It turns out that if a simple graph is to admit
a near-unanimity polymorphism, then it must be bipartite \cite{irreflex}.

The concept of \textit{absolute retracts} is intimately linked
to near-unanimity polymorphisms.  A graph is an absolute retract
with respect to isometry if it is a retract of each
graph in which it is isometrically embedded.
Absolute retracts and their relation to near-unanimity
polymorphisms have been studied for both reflexive and 
simple graphs - see for instance 
\cite{Reading-1, Reading-2, Reading-3, reflex, Bandelt ARI, Bandelt maj, Loten}.
In particular, Bandelt \cite{Bandelt maj} showed that a bipartite graph
admits a $3$-ary near-unanimity polymorphism if and only if it is an
absolute retract with respect to isometry.
For reflexive digraphs, the connection between the two concepts has been worked out for arbitrary arities;
a reflexive digraph admits a $(k+1)$-ary near-unanimity polymorphism if and only if
it is an absolute retract with respect to tree obstructions with at most $k$ leaves \cite{reflex, Loten}.

In this article, we establish a similar
characterization of bipartite graphs which admit a $k$-ary near-unanimity polymorphism 
thus mirroring the reflexive case and generalizing Bandelt's result for bipartite graphs.
The technique crucially relies upon the finite 
homomorphism duality result for strongly bipartite digraphs admitting
a near-unanimity polymorphism proven by Feder et al. \cite{irreflex}.

\section{Preliminaries}

A \textit{digraph} is a relational structure with only a single binary relation.
A digraph will be denoted with a blackboard font (e.g. $\bb{G}, \bb{H}$) while the corresponding
Latin equivalent (e.g. $G, H$) will be used to denote the vertex set.
If $\bb{H}$ is a digraph then $E(\bb{H})$ denotes the binary relation on $\bb{H}$.
A member of $E(\bb{H})$ is called an \textit{edge}.
An edge will typically be denoted with no parentheses - e.g. the edge
$(a,b)$ will be written simply as $ab$.

An \textit{undirected graph} is a digraph whose edge relation is symmetric.
If $\bb{H}$ is an undirected graph and $u \in H$ is a vertex, the \textit{degree of $u$},
denoted $deg(u)$, is the number of edges incident to $u$.
A \textit{strongly biparite digraph} is a digraph with
the property that its vertex set
may be partitioned into a set of sinks and a set of sources.
If a graph is bipartite then it is undirected.
A digraph is non-trivial if it has more than one vertex and at least one edge.
If $\bb{H}$ is a digraph, then $\bb{H}_u$ is the undirected graph
that results from making all edges in $\bb{H}$ symmetric.
A digraph $\bb{H}$ is \textit{connected} if there is a path
between any two vertices in $\bb{H}_u$.
If $a,b \in H$, the \textit{distance between $a$ and $b$ in $\bb{H}$} is the length
of a shortest path between $a$ and $b$ in $\bb{H}_u$ and will be denoted
by $d_{\bb{H}}(a,b)$.

A \textit{subgraph of} $\bb{H}$ is a digraph $\bb{H'}$ such that
$H' \subseteq H$ and $E(\bb{H'}) \subseteq E(\bb{H})$.
For this article, if $\bb{H}$ is an undirected graph
then any subgraph of $\bb{H}$ is assumed to also be an undirected graph.
In other words, when removing an edge from $\bb{H}$ to yield a proper subgraph,
both orientations of the edge must be removed.

Let $\bb{G}$ and $\bb{H}$ be digraphs.  A \textit{homomorphism}
from $\bb{G}$ to $\bb{H}$ is a map $f$ from $G$ to $H$ such that
if $ab$ is an edge in $\bb{G}$ then $f(a)f(b)$ is an edge in $\bb{H}$.
Suppose that $\bb{G}$ contains $\bb{H}$ as an induced subgraph.
$\bb{H}$ is a \textit{retract} of $\bb{G}$ if there exists a homomorphism
$r: \bb{G} \rightarrow \bb{H}$ that acts as the identity on $H$.
The map $r$ is called a \textit{retraction}.
If $\bb{H}$ is a retract of $\bb{G}$ we write $\bb{H} \ret \bb{G}$.

Let $\bb{H}_1, \dots, \bb{H}_k$ be digraphs.
The \textit{relational product of} $\bb{H}_1, \dots \bb{H}_k$ is the digraph
$\bb{H}_1 \times \dots \times \bb{H}_k$ with vertex set
$H_1 \times \dots \times H_k$ and an edge between $(a_1, \dots, a_k)$ and $(b_1, \dots, b_k)$
if and only if $a_ib_i \in E(\bb{H}_i)$ for each $i \in \Set{1, \dots, k}$.
The digraph $\bb{H}^k$ is the relational product of $k$ copies of $\bb{H}$.

\begin{defn}

	Let $\bb{H}$ be a digraph.  An
	\textit{$\bb{H}$-coloured digraph} is a digraph $\bb{G}$
	equipped with unary relations
	$S_h \subseteq G$ for every $h \in H$.  If $v \in S_h(\bb{G})$ then
	we say that $v$ is \textit{coloured with $h$}.
	$\bb{H}$ supports a natural $\bb{H}$-coloured digraph, denoted
	$\bb{H}^c$, where
	$S_h(\bb{H}) = \Set{h}$ for all $h \in H$. 
	If $\bb{G}$ contains $\bb{H}$ as a subgraph then
	we can colour each vertex $h \in H$ with $h$.
	Explicitly, let $\bb{G}_{\bb{H}}$ be the
	$\bb{H}$-coloured digraph where $S_h(\bb{G}) = \Set{h}$ for all
	$h \in H$ and whose underlying digraph is $\bb{G}$.
\end{defn}

Homomorphisms between $\bb{H}$-coloured digraphs are defined in the natural way.
Let $\bb{X}$ and $\bb{Y}$ be $\bb{H}$-coloured digraphs.  A map
$f: \bb{X} \rightarrow \bb{Y}$ is an $\bb{H}$-coloured digraph homomorphism
if it is a homomorphism of the underlying digraphs from $\bb{X}$ to $\bb{Y}$
and satisfies $f(S_h(X)) \subseteq S_h(Y)$ for each $h \in H$.
A \textit{substructure} of $\bb{X}$ is an $\bb{H}$-coloured digraph
$\bb{X'}$ whose underlying digraph is a 
subgraph of the underlying digraph
of $\bb{X}$ and $S_h(X') \subseteq S_h(X)$ for each $h \in H$.
$\bb{X'}$ is a \textit{proper substructure} if $\bb{X'}$ is a substructure
of $\bb{X}$ and $\bb{X'} \neq \bb{X}$.
The distance between two vertices in an $\bb{H}$-coloured digraph is the 
distance between the vertices in the underlying digraph.

\begin{defn}
	Let $\bb{H}$ be a digraph and let $\bb{X}$ and $\bb{Y}$ be $\bb{H}$-coloured digraphs.
	$\bb{X}$ is an \textit{obstruction for} $\bb{Y}$
	if $\bb{X}$ does not admit a homomorphism to $\bb{Y}$ and we write
	$\bb{X} \nrightarrow \bb{Y}$.
	If $\bb{X}$ is not an obstruction for $\bb{Y}$ then we say that
	$\bb{X}$ is \textit{feasible for} $\bb{Y}$ and write $\bb{X} \rightarrow \bb{Y}$.
	$\bb{X}$ is a \textit{critical obstruction} for $\bb{Y}$
	if it is an obstruction for $\bb{Y}$ yet any proper substructure of $\bb{X}$
	is feasible for $\bb{Y}$.
\end{defn}

Let $\bb{X}$ and $\bb{Y}$ be $\bb{H}$-coloured digraphs for some digraph $\bb{H}$.
If $\bb{X}$ is an obstruction for $\bb{Y}$, then there must be a substructure of $\bb{X}$
that is a critical obstruction for $\bb{Y}$.  Explicitly, among all substructures of $\bb{X}$
that are still obstructions for $\bb{Y}$, any minimal substructure is a critical obstruction for $\bb{Y}$.

We can rephrase the concept of a retraction in terms of
homomorphisms between $\bb{H}$-coloured digraphs.
If $\bb{G}$ contains $\bb{H}$ as an induced subgraph then 
$\bb{G}_{\bb{H}}$ is feasible for $ \bb{H}^c$ if and only if
$\bb{H}$ is a retract of $\bb{G}$.
There are simple necessary conditions for $\bb{H}$ to be a retract of $\bb{G}$.
For instance, $\bb{G}$ and $\bb{H}$ must have the same chromatic number,
and $\bb{H}$ must be isometrically embedded in $\bb{G}$ in the sense that
for any $a,b \in H$ it must be that $d_{\bb{H}}(a,b) = d_{\bb{G}}(a,b)$.
Each necessary condition for the existence of a retraction can be used
to define a class of graphs, the absolute retracts,
for whom the necessary condition is also sufficient.
The following definition gives rise to a necessary condition for 
the existence of a retraction
and thus a class of absolute retracts.

\begin{defn}
\leavevmode
\begin{enumerate}
\item
	Let $\bb{H}$ be a bipartite graph.
	$\bb{T}$ is an  \textit{$\bb{H}$-tree} if
	$\bb{T}$ is an $\bb{H}$-coloured tree 
	such that the leaves are exactly the
	coloured elements of $\bb{T}$.
	Further, we require that
	if $v_a, v_b \in T$ are leaves with colours $a, b$ respectively, then 
	$d_{\bb{T}}(v_a, v_b) \equiv d_{\bb{H}}(a,b) \pmod 2$.

\item
	Let $\bb{H}$ be a strongly bipartite digraph.
	We say that $\bb{T}$ is a \textit{directed
	$\bb{H}$-tree} if $\bb{T}$ is an $\bb{H}$-coloured
	strongly bipartite tree such that the leaves are exactly the coloured
	elements of $\bb{T}$.  Further,
	if a leaf $l \in T$ is coloured with $h \in H$ that is a sink (or source) in
	$\bb{H}$, then the unique edge incident to $l$
	must be incoming (resp. outgoing).

	In both cases we require that each leaf is uniquely coloured.
\end{enumerate}
\end{defn}

Let $\bb{G}$ be a bipartite graph and $\bb{H}$ a retract of $\bb{G}$.
Then any $\bb{H}$-tree $\bb{T}$ that is feasible for $\bb{G}_\bb{H}$ must also be feasible
for $\bb{H}^c$.  Indeed, if $r: \bb{G}_{\bb{H}} \rightarrow \bb{H}^c$ 
and $\phi: \bb{T} \rightarrow \bb{G}_{\bb{H}}$ are homomorphisms,
then certainly $r \circ \phi$ is a homomorphism from $\bb{T}$ to $\bb{H}^c$.
In other words, for $\bb{H}$ to be a retraction of $\bb{G}$, each
$\bb{H}$-tree obstruction for $\bb{H}^c$ must also be an obstruction for
$\bb{G}_{\bb{H}}$.  This is the necessary condition used to define
the following class of absolute retracts - with the caveat that we
restrict to $\bb{H}$-tree obstructions with a bounded number of leaves.

\begin{defn}
	Let $\bb{H}$ be a bipartite graph and $k \geq 2$.
	We say that $\bb{H}$ is a \textit{bipartite absolute retract with respect
	to $k$-trees} if whenever a bipartite graph $\bb{G}$ contains
	$\bb{H}$ as an induced subgraph such that all
	$\bb{H}$-tree obstructions for $\bb{H}^c$ with at most $k$ 
	leaves are also obstructions for $\bb{G}_{\bb{H}}$, then
	$\bb{H}$ is a retract of $\bb{G}$.
\end{defn}

The bipartite absolute retracts with respect to $2$-trees
are exactly the bipartite absolute retracts with respect to isometry.
Indeed, if a bipartite graph $\bb{H}$ is isometrically embedded in a 
bipartite graph $\bb{G}$, then all $\bb{H}$-tree obstructions for $\bb{H}^c$ with
$2$ leaves must also be obstructions for $\bb{G}_{\bb{H}}$.
The following result generalizes Bandelt's \cite{Bandelt ARI}
result for graphs which admit a $3$-NU polymorphism
to graphs which admit a $k$-NU polymorphism for arbitrary $k$, the proof of 
which appears at the end of the next section.

\begin{theorem}
\label{theThm}
	Let $\bb{H}$ be a connected, bipartite graph and $k \geq 2$.
	$\bb{H}$ is a bipartite absolute retract
	with respect to $k$-trees if and only if
	it admits a $(k+1)$-NU polymorphism.
\end{theorem}

\section{Trees and Duality}

\begin{defn}
\leavevmode
\begin{enumerate}
\item 
	Let $\bb{H}$ be a digraph.
	A \textit{homomorphism duality} for
	$\bb{H}^c$ is a set $\mathcal{F}$ of $\bb{H}$-coloured digraphs
	such that if $\bb{X}$ is an $\bb{H}$-coloured digraph, then
	$\bb{X} \nrightarrow \bb{H}^c$ if and only if
	there exists some $\bb{G} \in \mathcal{F}$ such that
	$\bb{G} \rightarrow \bb{X}$.  If $\bb{H}^c$ admits a finite homomorphism
	duality then we say $\bb{H}^c$ has \textit{finite homomorphism duality}. 

\item
	Let $\bb{H}$ be an undirected graph.
	An \textit{undirected homomorphism duality for $\bb{H}^c$}
	is a set $\mathcal{F}$ of $\bb{H}$-coloured undirected graphs
	such that if $\bb{X}$ is an $\bb{H}$-coloured undirected graph,
	then $\bb{X} \nrightarrow \bb{H}^c$ if and only if
	there exists some $\bb{G} \in \mathcal{F}$ such that
	$\bb{G} \rightarrow \bb{X}$.
\end{enumerate}
\end{defn}

Let $\bb{H}$ be a strongly bipartite digraph.
If an $\bb{H}$-coloured  digraph $\bb{X}$ were to admit a homomorphism to $\bb{H}^c$
then it must be strongly bipartite, have no vertex with multiple colours,
and it must be coloured consistently with $\bb{H}$ in the sense that
a sink (or source) in $\bb{X}$ cannot be coloured with a source (or sink) in $\bb{H}$.
Thus any homomorphism duality for $\bb{H}^c$
must include the following critical obstructions.

\begin{defn}[See {\cite{irreflex}}]
	Let $\bb{H}$ be a non-trivial, connected, strongly bipartite digraph
	with source set $D$ and sink set $U$.  The \textit{elementary
	critical obstructions} for $\bb{H^c}$, denoted by $\mathcal{D}(\bb{H}^c)$,
	are the following $\bb{H}$-coloured digraphs:
	\begin{enumerate}[label=(\Alph*)]
		\item For all distinct $a,b \in H$, a single vertex coloured with $a$ and $b$.
		\item For all $d \in D$, two vertices with a single (directed) edge joining them, with sink coloured with $d$.
		\item For all $u \in U$, two vertices with a single (directed) edge joining them, with source coloured with $u$.
		\item The directed path of length $2$ with no coloured vertices.
	\end{enumerate}
\end{defn}

Similarly, we can define elementary critical obstructions for
bipartite graphs.

\begin{defn}
	Let $\bb{H}$ be a non-trivial, connected, bipartite graph.
	The elementary critical obstructions for $\bb{H^c}$, denoted
	by $\mathcal{D}(\bb{H}^c)$, are the following $\bb{H}$-coloured undirected graphs:
	\begin{enumerate}[label=(\Alph*)]
		\item For all distinct $a,b \in H$, a single vertex coloured
			with $a$ and $b$.
		\item For all $a,b \in H$, a path of length $l$ joining
			a vertex coloured with $a$ to a vertex coloured
			with $b$ such that $d_{\bb{H}}(a,b) \not\equiv l \pmod 2$.
		\item All cycles of odd length with no coloured vertices.
	\end{enumerate}
\end{defn}

Note that the elementary critical obstructions for both
strongly bipartite digraphs and for bipartite graphs are all indeed
critical obstructions.  There are only a finite number of
elementary critical obstructions for strongly bipartite digraphs,
however this is not true for bipartite graphs.

The following result
will be key in proving Theorem \ref{theThm}.

\begin{lemma}[See {\cite[Theorem 3.1 and the proof of Theorem 4.2]{irreflex}}]
	Let $\bb{H}$ be a connected, strongly bipartite digraph.
	Then $\bb{H}^c$ has finite homomorphism duality if and only if
	$\bb{H}$ admits an NU polymorphism.  Further, if
	$\bb{H}$ admits a $(k+1)$-NU polymorphism, then $\bb{H}^c$ has
	a finite homomorphism duality consisting of $\mathcal{D}(\bb{H}^c)$
	and a finite number of directed $\bb{H}$-trees with at most $k$ leaves,
	each of which are critical obstructions for $\bb{H}^c$.
\end{lemma}

Let $\bb{H}$ be an undirected graph.
The following construction gives a correspondence
between $\bb{H}$-coloured undirected graphs and 
undirected graphs that contain $\bb{H}$ as an induced subgraph.
Let $\bb{G}$ be an $\bb{H}$-coloured
undirected graph such that no vertex has more than one colour
and there are no edges between coloured vertices
whose colours are not adjacent in $\bb{H}$. 
Define a graph that contains $\bb{H}$ as an induced subgraph called the
\textit{$\bb{H}$-embed} of $\bb{G}$, denoted $\bb{G}^{\emb}$, as follows.

Let $U$ be the
set of non-coloured vertices of $\bb{G}$.
$\bb{G}^{\emb}$ has vertex set $H \cup U$ and
edges on each part as they are in their original graph.
Moreover there is a (symmetric) edge between $h \in H$ and $u \in U$ if
$u$ is adjacent in $\bb{G}$ to some $c \in S_h(\bb{G})$.

\begin{lemma}[See {\cite[Lemma 5.6]{reflex}}]
	Let $\bb{H}$ be an undirected graph and let $\bb{G}$ be an
	$\bb{H}$-coloured undirected graph that satisfies the requirements
	of the above construction.
	Then $\bb{G} \rightarrow \bb{H}^c$ if and only if
	$\bb{H} \ret \bb{G}^{\emb}$.
\end{lemma}

The following lemma allows us to leverage the power
of Lemma 1 for use on bipartite graphs.

\begin{lemma}[See {\cite[Lemma 5.1]{irreflex}}] 
	Let $\bb{H}$ be a bipartite graph with partition $A$ and $B$.
	Let $\bb{H}_d$ be the strongly bipartite digraph
	constructed from $\bb{H}$ by directing all the edges from $A$ to $B$
	(or from $B$ to $A$).
	Then $\bb{H}$ admits a $k$-NU polymorphism if and only if $\bb{H}_d$
	admits a $k$-NU polymorphism.
\end{lemma}

\begin{proof}[Proof of Theorem \ref{theThm}]

	($\Rightarrow$)
	The argument for this implication is adapted from \cite[Lemma 3.5]{Loten}.
	Let $\bb{H}$ be a connected, bipartite absolute retract with respect to $k$-trees.
	Let $\bb{K}$ be the $\bb{H}$-coloured undirected graph
	with underlying graph $\bb{H}^{k+1}$ and each
	near-unanimous (constant, except perhaps for one coordinate) vertex coloured with its near-unanimous value.
	Then $\bb{H}$ admits a $(k+1)$-NU polymorphism if and only if
	$\bb{K} \rightarrow \bb{H}^c$, which by Lemma 2 is true
	if and only if $\bb{H} \ret \bb{K}^{\emb}$.
	Note that $\bb{K}^{\emb}$ contains $\bb{H}$ as an induced subgraph
	and is bipartite.

	Suppose that $\bb{H}$ is not a retract of $\bb{K}^{\emb}$.
	Then there is an $\bb{H}$-tree $\bb{T}$ with at most
	$k$ leaves that admits a homomorphism
	$\phi: \bb{T} \rightarrow \bb{K}_{\bb{H}}^{\emb}$, 
	but is an obstruction for $\bb{H}^c$.
	Since there is a substructure of $\bb{T}$ that is a critical obstruction for $\bb{H}^c$ and
	is again an $\bb{H}$-tree, 
	we may assume that $\bb{T}$ is a critical obsutrction for $\bb{H}^c$.
	Moreover, we may assume that $\bb{T}$ is minimal in the following sense:
	among all critical $\bb{H}$-tree obstructions of $\bb{H}^c$ with
	at most $k$ leaves 
	which are feasible for $\bb{K}_{\bb{H}}^{\emb}$,
	$\bb{T}$ has the least number of vertices.

	Suppose that
	$u \in T$ is a non-coloured vertex but $\phi(u) \in \bb{K}^{\emb}_{\bb{H}}$ is coloured with some $h \in H$.
	Split $\bb{T}$ into $deg(u)$ $\bb{H}$-trees by removing $u$
	from $\bb{T}$ and then adding $u$ back as a leaf coloured with $h$ to each
	of the resulting trees.  At least one of the resulting trees,
	say $\bb{T}'$, must be an obstruction for $\bb{H}^c$.
	This contradicts the minimality of $\bb{T}$, since
	there is a substructure of $\bb{T}'$ that is a critical
	obstruction of $\bb{H}^c$ yet still admits a homomorphism
	to $\bb{K}_{\bb{H}}^{\emb}$.  Thus the image under $\phi$ of any
	non-coloured vertex of $\bb{T}$ is non-coloured in $\bb{K}^{\emb}_{\bb{H}}$.

	Let $l_1, \dots, l_k$ be the leaves of $\bb{T}$.
	For each $i \in \Set{1, \dots, k}$, let $h_i \in H$ be the colour of $l_i$ and let $u_i \in T$ be its unique neighbour.
	Then $\phi(u_i) \in H^{k+1}$ is adjacent to the unique vertex coloured with $h_i$ in $K^{\emb}_{\bb{H}}$.
	Therefore, at least $k$ of the coordinates of $\phi(u_i)$
	are adjacent to $h_i$ in $\bb{H}$.  Thus there is some index $p \in \Set{1, \dots k+1}$ 
	such that for every $i \in \Set{1, \dots, k}$, 
	$h_i\pi_p(\phi(u_i))$ is an edge of $\bb{H}$, where $\pi_p$ is the $p-th$ projection.
	Thus the map from $\bb{T}$ to $\bb{H}^c$ that sends leaves to their colour and sends
	non-leaves $x$ to $\pi_p(\phi(x))$ is a homomorphism.  This contradicts
	the fact that $\bb{T}$ is an obstruction for $\bb{H}^c$. 
	Therefore any $\bb{H}$-tree obstruction
	for $\bb{H}^c$ must also be an obstruction for $\bb{K}^{\emb}$,
	hence $\bb{H} \ret \bb{K}^{\emb}$.

	($\Leftarrow$) Let $\bb{H}$ be a bipartite graph that admits
	a $(k+1)$-NU polymorphism.  
	Let $\bb{H}_d$ be the strongly
	bipartite digraph made from $\bb{H}$ by directing the edges
	(in one of the two ways).
	$\bb{H}_d$ still admits a $(k+1)$-NU polymorphism by Lemma 3.
	Thus by Lemma 1, $\bb{H}^c_d$
	has a finite homomorphism duality consisting of $\mathcal{D}(\bb{H}^c_d)$ and a finite
	set $\mathcal{T}$ of directed $\bb{H}$-trees 
	with at most $k$ leaves which are critical obstructions for $\bb{H}^c_d$.
	Let $\mathcal{T}_u = \Set{\bb{T}_u : \bb{T} \in \mathcal{T}}$.

	We claim that $\mathcal{D}(\bb{H}^c) \cup \mathcal{T}_u$ is an
	undirected homomorphism
	duality for $\bb{H}^c$.  Certainly all members of the proposed
	duality are obstructions for $\bb{H}^c$.
	Suppose that $\bb{X}$ is an $\bb{H}$-coloured undirected graph
	that is an obstruction for $\bb{H}^c$.
	We may assume that no member of $\mathcal{D}(\bb{H}^c)$ admits a homomorphism to
	$\bb{X}$.
	Then $\bb{X}$ is bipartite and each coloured vertex is uniquely coloured.
	Moreover, since no elementary critical obstruction
	of type (B) admits a homomorphism to $\bb{X}$,
	we may direct the edges of $\bb{X}$ to yield a strongly bipartite
	digraph $\bb{X}_d$ that is consistently $\bb{H}_d$-coloured, in the sense
	that if a coloured vertex of $\bb{X}_d$ is a source (or sink) then its
	colour is a source (or sink) in $\bb{H}_d$.
	Since $\bb{X}_d$ is an obstruction for $\bb{H}_d$ and no
	elementary critical obstructions for $\bb{H}_d$ admit
	a homomorphism to $\bb{X}_d$, it must be that some $\bb{T} \in \mathcal{T}$
	is feasible for $\bb{X}_d$.  Thus $\bb{T}_u$ is feasible for
	$\bb{X}$, verifying the claim.

	Let $\bb{G}$ be a bipartite graph that contains $\bb{H}$
	as an induced subgraph and all $\bb{H}$-tree
	obstructions for $\bb{H}^c$ with at most $k$ leaves
	are also obstructions of $\bb{G}_{\bb{H}}$.
	Suppose that $\bb{G}_{\bb{H}} \nrightarrow \bb{H}^c$.
	Then since no elementary critical obstruction of
	$\bb{H}$ admits a homomorphism to $\bb{G}_{\bb{H}}$, it must be that
	some $\bb{T} \in \mathcal{T}_u$ admits a homomorphism to $\bb{G}_\bb{H}$.
	This is a contradiction since $\bb{T}$ is an $\bb{H}$-tree obstruction
	with at most $k$ leaves.  Thus it must be that $\bb{H} \ret \bb{G}$.

\end{proof}

\section{Further problems}

Bandelt et al. were able to characterize the internal strucutre of
bipartite graphs which admit a $3$-NU polymorphism using
the notion of the \textit{interval} between two vertices of a graph \cite{Bandelt ARI}.
Given a bipartite graph $\bb{H}$ and two vertices $u,v \in H$,
the \textit{interval between u and v} is the set of all vertices on a shortest 
path between $u$ and $v$,
\[
	I(u,v) = \Set{x \in H: d(x,u) + d(x,v) = d(u,v)}.
\]
A bipartite graph $\bb{H}$ admits a $3$-NU polymorphism if and only if
for all $u,v \in H$ with $d_{\bb{H}}(u,v) \geq 3$ the neighbours of $u$ in $I(u,v)$
all share a second common neighbour in $I(u,v)$.
A similar internal characterization of bipartite graphs which
admit an NU polymorphism of larger arities remains to be found.

\section*{Acknowledgements}
I'm thankful for the guidance that my supervisor, Ross Willard, 
generously provided throughout the course of this research.
I would like to acknowledge the support
provided by the University of Waterloo and the National Science
and Engineering Research Council of Canada.

\end{document}